\documentclass[10pt,twoside,a4paper,reqno]{amsart}

\usepackage{amsfonts,amsthm,latexsym,mathtools,amsmath,amssymb,enumerate,appendix,color,hyperref,epsf}
\usepackage{bbm}
\usepackage[capitalize]{cleveref}
\usepackage{float}
\usepackage{xcolor, graphicx, tikz, tikz-cd}
\usepackage{graphicx}
\usepackage{subcaption}

\usepackage[utf8]{inputenc}
\usepackage{amsmath} 
\usepackage{amssymb}
\usepackage{amsthm}

\newtheorem{theorem}{Theorem}
\numberwithin{theorem}{section}
\newtheorem{prop}{Proposition}
\numberwithin{prop}{section}
\newtheorem{lemma}{Lemma}
\numberwithin{lemma}{section}
\newtheorem{corollary}{Corollary}
\numberwithin{corollary}{section}
\newtheorem{definition}{Definition}
\numberwithin{definition}{section}
\newtheorem{conj}{Conjecture}
\numberwithin{conj}{section}
\theoremstyle{remark}
\newtheorem{remark}{Remark}
\numberwithin{remark}{section}
\newtheorem{example}{Example}
\numberwithin{example}{section}
\newtheorem*{acknow}{Acknowledgement}
\numberwithin{equation}{section}

\makeatletter
\@namedef{subjclassname@2020}{\textup{2020} Mathematics Subject Classification}
\makeatother

\newcommand{\real}{\mathbb{R}}

\newcommand{\complex}{\mathbb{C}}
\newcommand{\nat}{\mathbb{N}}
\newcommand{\integ}{\mathbb{Z}}

\renewcommand{\Re}{\operatorname{Re}}
\renewcommand{\Im}{\operatorname{Im}}
\renewcommand{\sp}{\operatorname{sp}}

\begin{document}
\title{On reality of eigenvalues of banded block Toeplitz matrices}

\author[Dario Giandinoto]{Dario Giandinoto}
\address{Department of Mathematics, Stockholm University, SE-106 91 Stockholm,
      Sweden}
\email{dario.giandinoto@math.su.se}

\date{\today}


\begin{abstract} We formulate and partially prove a general conjecture providing necessary and sufficient conditions for the reality of the asymptotic spectrum of an arbitrary real banded block Toeplitz matrix. Additionally we present numerical experiments supporting it. This conjecture is a direct generalization of the already existing one in the case of banded Toeplitz matrices.
\end{abstract}

\maketitle

\section{Introduction} \label{sec:intro} 

The motivation for this study is twofold. Firstly,  in \cite{ShSt} B. Shapiro and F. \v{S}tampach considered the question when every  principal submatrix of a real banded infinite Toeplitz matrix  $T$ has real spectrum. The main claim of this paper is that the necessary and sufficient condition for the reality of the sequence of spectra is the existence of a Jordan curve in the pull-back $b^{-1}(\real)$ which contains the origin in its interior, see Theorem 1 of \cite{ShSt}.  Here $b(z)$ is the symbol of the banded Toeplitz matrix $T$. (The definition of a symbol is given in \eqref{eq:symbol}). 

\medskip
Unfortunately, a gap was later found in the proof of the sufficient condition, see \cite{ShSt2}, and currently this statement is only a conjecture supported by a large numerical evidence  and  proven in the necessary direction.  Extending the set-up of \cite{ShSt} from the Toeplitz to the block Toeplitz scenario is our first motivation.  Notice that different aspects of block Toeplitz matrices  have been extensively  studied since the pioneering papers \cite{Wi1, Wi2}. 

 \medskip
 The second motivation has its origin in physics. Different   periodic $1$-dimensional  structures appearing in natural sciences are typically presented  by   infinite banded matrices whose entries are invariant under an appropriate diagonal shift (called the \emph{sublattice step} in the relevant physics literature). The simplest example of this situation occurs when the structures are periodic with step  $k=1$ in which case one deals with the usual banded Toeplitz matrices. For the sublattice step $k>1$ one encounters block Toeplitz matrices with Toeplitz blocks of size  $k\times k$ (also called $k$-Toeplitz matrices). 


\medskip
Recent years witness the explosion of the interest in non-Hermitian physics (which deals with open system when the Hamiltonian doesn't have to be represented by Hermitian matrices), see e.g. \cite{ABL,  ABdBLT}. Spectral properties of operators for non-Hermitian situation are very different from that of the conventional physics. Especially, non-Hermitian systems having real spectrum are of fundamental  interest both theoretically and for practical applications, see \cite{ABL, ABdBLT, CZ, BPP, KS, YLFWZ}. Therefore finding conditions for reality of spectrum for block Toeplitz matrices is currently a very relevant topic.


\medskip 
Our set-up is as follows. 
Let us consider a sequence of $k \times k$ matrices $(B_m)_{m \in \integ}$. It can be encoded in a matrix-valued function
\[ B(z) = \sum_{m=-\infty}^{\infty} B_m z^m. \]
To such a function $B: \complex \rightarrow \complex^{k \times k}$ and for $n \in \nat$ we associate the block Toeplitz matrix
\[ T_n(B) = (B_{i-j})_{i,j=1}^n ~ \in \complex^{nk \times nk}. \]
We say that $B$ is the \emph{symbol} of $T_n(B)$. Explicitly
\[ T_n(B) = \begin{pmatrix}
B_0 & B_{-1} & \cdots & B_{-n+1} \\
B_1 & B_0 & \ddots & \vdots\\
\vdots & \ddots & \ddots & B_{-1} \\
B_{n-1} & \cdots & B_1 & B_0
\end{pmatrix}.
\]
We will also denote by $T(B)$ the semi-infinite block Toeplitz matrix 
\[ T(B) = (B_{i-j})_{i,j=1}^{\infty}. \]
Since we are interested in the case of banded matrices, we assume $B(z)$ is a matrix-valued Laurent polynomial, i.e. there exist positive integers $s$ and $r$ such that
\begin{equation}
\label{eq:blocksymb}
B(z) = \sum_{m= - r}^s B_m z^m.
\end{equation}
(We assume $r,s \geq 1$ since in other cases the spectrum is just given by the diagonal block $B_0$.) 

\medskip
Define the \emph{scalar symbol} $b: \complex \rightarrow \complex$ of $T(B)$ as given by  
$b(z):=\det B(z)$. 
Since every entry of $B(z)$ is a Laurent polynomial then $b(z)$ is also a Laurent polynomial in $z$. In other words, there exist integers $p,q>0$ satisfying $p \leq ks$ and $q \leq kr$, such that
\begin{equation}
\label{eq:symbol}
b(z) = \sum_{m=-q}^p b_m z^m, 
\end{equation} 
where we assume that $b_{-q}$ and $b_p$ are non-zero.

Finally,  define the bivariate \emph{characteristic function} $f(z,\lambda):= \det (B(z)-\lambda I_k)$. As above, we can write $f$ as
\[ f(z,\lambda) = \sum_{m=-q}^p f_m(\lambda) z^m, \]
where each $f_m$ is a polynomial in $\lambda$ of degree at most $k$. \par
Given $n \in \nat$, we denote by $\sp(T_n(B))$ the spectrum of the matrix $T_n(B)$:
\[ \sp(T_n(B)) = \{ \lambda \in \complex: \det(T_n(B)- \lambda I_{nk})=0 \}. \]
Our main object of study is the \emph{asymptotic spectrum} $\Lambda(B)$ of $T(B)$ defined as  
\[ \Lambda(B)= \left\{ \lambda \in \complex: \lim_{n \rightarrow \infty} \text{dist}(\lambda,\sp(T_n(B)))=0 \right\}. \]

\medskip
Below we study the question of reality of $\Lambda(B)$ for  real banded block Toeplitz matrices $T_n(B)$. Let us first present the important general description of $\Lambda(B)$ given by H.~Widom in \cite{Wi1,Wi2} and later clarified by S.~Delvaux \cite{De}. (These results are similar in spirit to the earlier theorem of Schmidt-Spitzer dealing with the Toeplitz case, see \cite{SnSp} and Chapter 11 of \cite{BoeGr}). 

\medskip
Let $z_1(\lambda), \dots, z_{p+q}(\lambda)$ be the roots of the polynomial equation $z^q f(z,\lambda)=0$ with respect to $z$, ordered according to their moduli, i.e. 
\begin{equation}
\label{eq:zjl}
 0< \vert z_1(\lambda) \vert \leq \vert z_2(\lambda) \vert \leq \dots \leq \vert z_{p+q}(\lambda) \vert.
\end{equation}

Define the set 
\[ \Lambda_0(B) := \{\lambda \in \complex: \vert z_q(\lambda) \vert = \vert z_{q+1}(\lambda) \vert \}. \]
Furthermore, we define 
\[ C_0(B,\lambda) := \det \left( \frac{1}{2 \pi i} \int_{\sigma_0} z^{\mu - \nu} (B(z)- \lambda I_k)^{-1} \frac{dz}{z} \right)_{\mu,\nu=1,\dots,r}, \]
where $r$ is given by (\ref{eq:blocksymb}) and $\sigma_0$ is a counterclockwise oriented closed Jordan curve enclosing $z=0$ together with the points $z_j(\lambda)$, $j=1, \dots, q$, but not enclosing other roots of $f(z,\lambda)=0$. Define
\[ G_0(B)=\{ \lambda \in \complex \setminus \Lambda_0(B): C_0(B,\lambda)=0\}. \]
With this notation, we have the following result (Theorem 3.1 in \cite{De}):

\begin{theorem}
\label{thm:delvaux}
Let $B(z): \complex \rightarrow \complex^{k \times k}$ be a matrix-valued symbol such that the following two conditions are satisfied:
\begin{enumerate}[(a)]
\item The set $\Lambda_0(B)$ is a subset of $\complex$ of two-dimensional Lebesgue measure 0;
\item The set $G_0(B)$ has finite cardinality.
\end{enumerate}
Then
\[ \Lambda(B) = \Lambda_0(B) \cup G_0(B). \]
\end{theorem}

Delvaux also proves that the hypotheses (a) and (b) of Theorem~\ref{thm:delvaux} are satisfied in the two following cases (Proposition 1.1 in \cite{De}):

\begin{itemize}
\item The set $\complex \setminus \Lambda_0(B)$ is connected and  $\Lambda_0(B)$  has no interior points. 
\item $T(B)$ is a Hessenberg matrix.
\end{itemize}

\medskip

In this paper we explore conditions under which $\Lambda_0(B) \subseteq \real$, with $T(B)$ being a real banded block Toeplitz matrix. Let us formulate a general conjecture concerning the reality of $\Lambda_0(B)$.
 Define the  set:
\[ \Gamma(B)=\{z \in \complex: \exists \lambda \in \real \text{ s.t. } f(z,\lambda)=0\}. \]

Notice how, if $k=1$ and $b(z)$ is a Laurent polynomial, the characteristic function reduces to
\[ f(z,\lambda)= b(z) - \lambda. \]
Then, the curve $\Gamma(b)$ in the scalar case (that is, for a banded Toeplitz matrix), coincides with $b^{-1}(\real)$, which is the set studied in \cite{ShSt}.

\begin{conj}
\label{conj:main} 
Given a real ($k \times k$)-matrix valued symbol $B(z)$, the set $\Lambda_0(B)$ of $T(B)$ belongs to  $\real$ if and only if $\Gamma(B)$ contains $k$ Jordan curves (with multiplicity) having $0$ in their interior.
\end{conj}

The main result of this paper is a proof of the necessary condition in Conjecture \ref{conj:main}. It is essentially an extension of Theorem 5 of \cite{ShSt} to the block Toeplitz case. In order to give the statement, we recall the following definition:

\begin{definition}
Let 
\[ f(x_1,\dots,x_n) = \sum_{\underline{\alpha} \in \integ^n} c_{\underline{\alpha}} x_1^{\alpha_1} \cdots x_n^{\alpha_n} \]
be a multi-variable (Laurent) polynomial, with $\underline{\alpha}=(\alpha_1,\dots,\alpha_n)$. Then we define the Newton polytope of $f$ to be the convex hull in $\real^n$ of the finite set
\[  \{ \underline{\alpha} \in \integ^n : c_{\underline{\alpha}} \neq 0 \}. \]
In the case $n=2$, this is also called the Newton polygon.
\end{definition}

We are now ready to state our main result.

\begin{theorem}
\label{thm:main}
Assume $B(z)$ is a matrix-valued Laurent polynomial as in \eqref{eq:blocksymb}, and its determinant $b(z)$ is a Laurent polynomial as in \eqref{eq:symbol}, with $b_{-q}$, $b_p \neq 0$. Furthermore, assume that the Newton polygon of $f(z,\lambda) = \det(B(z) - \lambda I)$ is a triangle with vertices $(-q,0)$, $(p,0)$ and $(0,k)$. \par
Under these hypotheses, if $\Lambda_0(B) \subset \real$, then $\Gamma(B)$ contains $k$ Jordan curves (with multiplicity) having $0$ in their interior.
\end{theorem}

\begin{remark}
We will show in Corollary \ref{cor:triangle} that the hypothesis about the Newton polygon of $f(z,\lambda)$ holds generically, i.e. if there are no cancellations between coefficients which lower the degree (or raise the order) of $b(z)$.
\end{remark}

\medskip

In Section 2 of this article, we compare the structure of $\Gamma(B)$ with what's already known in the case $k=1$. Then we formulate and prove two results which are necessary to prove Theorem \ref{thm:main}. Finally, in Section 3 we present some numerical experiments which illustrate the main result and explain why we believe that Conjecture \ref{conj:main} holds.

\section{Main result}
\label{sec:mres}

The goal of this section is to prove Theorem \ref{thm:main}. The proof follows a very similar reasoning to the $k=1$ case, but it just needs some additional technical details. \par

\medskip

\begin{figure}[b]
\centering
\includegraphics[width=0.45\textwidth]{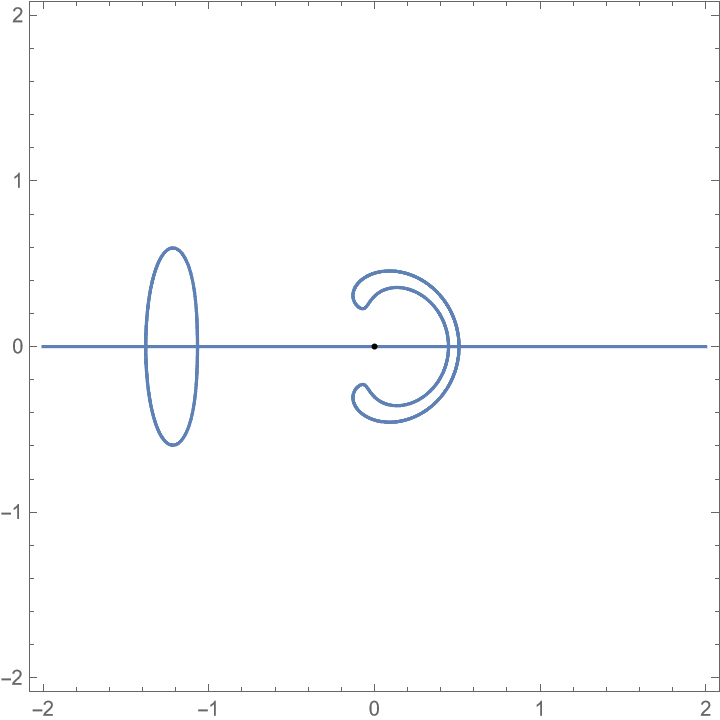}
\caption{Plot of $\Gamma(B)$ with $B$ as in \eqref{eq:gammab}}
\label{fgr:jordan}
\end{figure}

First we want to say some words about the structure of the set $\Gamma(B)$, since it differs from the case $k=1$ in some aspects. If we denote $z=x+iy$ and $\lambda=\alpha+i\beta$ we have that $\Gamma(B)$ is defined by the system of equations
\[ \begin{cases}
\Re f(x+iy,\alpha+i\beta)=0 \\
\Im f(x+iy,\alpha+i\beta)=0 \\
\beta=0
\end{cases} \]
One can eliminate the variable $\alpha$ in the system above, to obtain that there exists a bivariate polynomial $g(x,y)$ such that 
\[ \Gamma(B) = \{z=x+iy \in \complex: g(x,y)=0\}. \]
As we mentioned before, the set $\Gamma(B)$ is a generalization of $n_b = b^{-1}(\real)$, which was studied by Shapiro and $\check{\text{S}}$tampach in \cite{ShSt}. In their article they point out some properties of this set, which they call \emph{net of the rational function $b$}:
\begin{itemize}
\item \emph{No arc of $n_b$ can terminate in $\complex \setminus \{0\}$.} \\
This is true also for our set: as we showed above, the set $\Gamma(B)$ is a real algebraic curve in the variables $x$ and $y$, with $z=x+iy$. It follows that any arc in $\Gamma(B)$ has to be connected with $\infty$, or close into a loop.
\item \emph{Any Jordan curve contained in $n_b \setminus \{0\}$ must have 0 in its interior.} \\
This is no longer true in the case $k>1$. In Figure \ref{fgr:jordan} we plot $\Gamma(B)$ with
\begin{equation}
\label{eq:gammab}
B(z) = \begin{pmatrix}
-\frac{2}{z} + 13 + 3z & -5 -8z \\
-\frac{4}{z}-2 & \frac{1}{z}+5-6z
\end{pmatrix}.
\end{equation}
One can see how the set contains two Jordan curves which do not have $0$ in their interior. This is also a counterexample for the next property.
\item \emph{If $n_b$ contains a Jordan curve entirely located in $\complex \setminus \{0\}$, it's unique.} \\
As one can see from the same example, this no longer holds in our case.
\end{itemize}

\medskip

Now we go on to prove some properties of the characteristic function that will allow us to prove our main result. First we introduce some notation for the entries of our matrix, which will facilitate proving our result about the general coefficients of $f(z,\lambda)$. Let $A$ be a semi-infinite block Toeplitz matrix, with blocks of size $k$. This is equivalent to saying that the entries of $A=(A_{i,j})_{i,j=1}^{\infty}$ have the property 
\begin{align*}
A_{i,j} = A_{i+k,j+k} && \forall i,j \in \nat.
\end{align*}
In other words, the entries on a fixed diagonal of $A$ repeat periodically with period $k$. It follows that the entries of $A$ are determined by $k$ sequences $(a_{n,1})_{n \in \integ}, \dots, (a_{n,k})_{n \in \integ}$, simply by putting $A_{i,j}=a_{i-j,\min(i,j) (\text{mod }k)}$. With this notation, the entries of the $(i-j)$-th diagonal will be $a_{i-j,1},\dots,a_{i-j,k}$ and then they repeat periodically. For instance, if $k=3$, the matrix will look like:
\[ A=
\begin{pmatrix}
a_{0,1} & a_{-1,1} & a_{-2,1} & a_{-3,1} & a_{-4,1} & a_{-5,1} & \cdots \\
a_{1,1} & a_{0,2} & a_{-1,2} & a_{-2,2} & a_{-3,2} & a_{-4,2} & \cdots \\
a_{2,1} & a_{1,2} & a_{0,3} & a_{-1,3} & a_{-2,3} & a_{-3,3} & \cdots \\
a_{3,1} & a_{2,2} & a_{1,3} & a_{0,1} & a_{-1,1} & a_{-2,1} & \cdots \\
a_{4,1} & a_{3,2} & a_{2,3} & a_{1,1} & a_{0,2} & a_{-1,2} & \cdots \\
a_{5,1} & a_{4,2} & a_{3,3} & a_{2,1} & a_{1,2} & a_{0,3} & \cdots \\
\vdots & \vdots & \vdots & \vdots & \vdots & \vdots & \ddots
\end{pmatrix}.
\]

We recall that for a Laurent polynomial $f(x) = \sum_i c_i x^i$, we call 
\begin{align*}
\operatorname{ord}(f) = \min \{i: c_i \neq 0\} && \text{and} && \deg(f)= \max \{i: c_i \neq 0 \} 
\end{align*}
respectively the \emph{order} and \emph{degree} of $f$. \par
The first result we prove is a computation of the order and degree of the coefficients of the bivariate Laurent polynomial $f(z,\lambda)$.

\begin{prop}
\label{prop:newton}
Let $B(z):\complex \setminus\{0\} \rightarrow \complex^{k \times k}$ be a matrix-valued Laurent polynomial. Let us denote by $A=T(B)$ its associated block Toeplitz matrix. Assume $A_{i,j}=0$ for $i-j<-u$ and $i-j>v$, where $u=kr$ and $v=ks$ ($r,s$ as in \eqref{eq:blocksymb}) denote the bandwidth of the matrix $A$. As we did before, we define $f(z,\lambda):=\det(B(z)-\lambda I_k)$. We can write
\[ f(z,\lambda) = \sum_{\ell=0}^k g_{\ell}(z) \lambda^{\ell}, \]
where each $g_\ell(z)$ is a Laurent polynomial. Then we have generically
\begin{align*}
\operatorname{ord}(g_0) = -u, && \deg (g_0) = v,
\end{align*}

and for $\ell= 1, \dots, k$ we have

\begin{align}
\label{eq:ordg}
\operatorname{ord}(g_{\ell}) \geq \left \lceil \frac{-u(k-\ell)}{k} \right \rceil, && \deg(g_{\ell}) \leq \left \lfloor \frac{v(k-\ell)}{k} \right \rfloor.
\end{align}

\begin{proof}
With the notation introduced above, the matrix $A$ can be determined by $k$ (finite) sequences $(a_{n,1})_{-u \leq n \leq v}, \dots, (a_{n,k})_{-u \leq n \leq v}$. Then we can find explicit expressions for the entries of the symbol $B(z)$ in terms of $a_{n,1},\dots, a_{n,k}$. For a natural number $m$, all entries of the matrix coefficient of $z^m$ in $B(z)$ will be of the type $a_{km+i-j,\min(i,j) (mod k)}$, with $-k < i-j < k$. But for any non-zero $(a_{n,k})$, we know that $-u \leq n \leq v$. Then we have the two conditions on $m$:
\begin{gather*}
km+i-j \geq -u \Rightarrow m \geq \left \lceil \frac{-u-i+j}{k} \right \rceil, \\
km+i-j \leq v \Rightarrow m \leq \left \lfloor \frac{v-i+j}{k} \right \rfloor.
\end{gather*}
We can now obtain the explicit expression for any entry of $B(z)$. For $1 \leq i,j \leq k$, we have
\begin{equation}
\label{eq:bz}
(B(z))_{i,j} = \sum_{m=\lceil \frac{-u-i+j}{k} \rceil}^{-1} a_{km+i-j,i} z^m + a_{0,\min(i,j)} \mathbbm{1}_{[-u,v]}(i-j) + \sum_{m=1}^{\lfloor \frac{v-i+j}{k} \rfloor} a_{km+i-j,j} z^m .
\end{equation}
The expression above says that the $(i,j)$-th entry of a generic $B(z)$ is a Laurent polynomial of order $\left \lceil \frac{-u-i+j}{k} \right \rceil$ and degree $\left \lfloor \frac{v-i+j}{k} \right \rfloor$. To prove our Lemma, we need to understand the non-zero coefficients of $f(z,\lambda)=\det(B(z)- \lambda I_k)$. To do this, we use the formula for the determinant of a matrix $M=(m_{i,j})_{i,j=1}^k$:
\begin{equation}
\label{eq:det}
\det(M) = \sum_{\sigma \in S_k} \left( \text{sgn}(\sigma) \prod_{i=1}^k m_{i,\sigma(i)} \right),
\end{equation}
where $S_k$ denotes the group of all permutations on $\{1,\dots,k\}$. \par 
Now we want to prove the formula for $\deg(g_{\ell})$ in (\eqref{eq:ordg}). Consider the case $\ell=0$, we have $g_0(z)=\det(B(z))$, so it can be obtained explicitly by combining  \eqref{eq:bz} and \eqref{eq:det}. Then the degree of $\det(B(z))$ is equal to
\[ \max_{\sigma \in S_k} \deg \left( \prod_{i=1}^k z^{\lfloor \frac{v-i+\sigma(i)}{k} \rfloor} \right) = \max_{\sigma \in S_k} \sum_{i=1}^k \left \lfloor \frac{v-i+\sigma(i)}{k} \right \rfloor. \]
Proving (\ref{eq:ordg}) for $\ell=0$ means proving that $\deg(g_0(z))=v$. Then we want to show that
\begin{equation}
\label{eq:sumq}
\max_{\sigma \in S_k} \sum_{i=1}^k \left \lfloor \frac{v-i+\sigma(i)}{k} \right \rfloor = v.
\end{equation}
To do this, first observe that for any $\sigma \in S_k$, we have
\[ \sum_{i=1}^k \left( \frac{v-i+\sigma(i)}{k} \right)= \frac{1}{k} (kv - 1 - \dots - k + \sigma(1) + \dots + \sigma(k))= \frac{kv}{k}=v, \]
hence
\[ \max_{\sigma \in S_k} \sum_{i=1}^k \left \lfloor \frac{v-i+\sigma(i)}{k} \right \rfloor \leq v. \]
Now to prove \eqref{eq:sumq} it suffices to find some $\sigma \in S_n$ such that the sum is equal to $v$ . Since we want that $v-i+\sigma(i) \equiv 0~ (\text{mod }k)$, it suffices to choose $\sigma(i)=i-v~(\text{mod }k)$ for all $i \in \{1, \dots,k \}$, choosing $k$ instead of $0$ as a representative for the class $0$ (mod $k$). Using such a $\sigma$, we obtain
\[ \sum_{i=1}^k \left \lfloor \frac{v-i+\sigma(i)}{k} \right \rfloor = \sum_{i=1}^k \left( \frac{v-i+\sigma(i)}{k} \right) = v, \]
and this proves \eqref{eq:sumq}. \par
Now let us fix $\ell \in \{1, \dots, k\}$, then $g_{\ell}(z)$ is given by the sum of all principal minors of size $k-\ell$ of the matrix $B(z)$. Hence proving (\eqref{eq:ordg}) is equivalent to proving the inequality:
\begin{equation}
\label{eq:sumqke}
\max_{\underline{j} \in \mathcal{J}_{\ell}} \max_{\sigma \in S_{k-\ell}^{(\underline{j})}} \sum_{\substack{i=1 \\ i \neq j_1,\dots,j_{\ell}}}^k \left \lfloor \frac{v-i+\sigma(i)}{k} \right \rfloor \leq \left \lfloor \frac{v(k-\ell)}{k} \right \rfloor. 
\end{equation}
In the equation above
\[ \mathcal{J}_{\ell} = \{ \underline{j}=(j_1,\dots,j_{\ell}) \in \{ 1,\dots,k\}^{\ell}:~j_1 < j_2 < \cdots < j_{\ell} \} \]
and
\[ S_{k-\ell}^{(\underline{j})}= \{\sigma \in S_k : \sigma(j_{\alpha})=j_{\alpha}~~\forall \alpha=1,\dots,\ell \}. \]
As we did before, we compute the sum without the floor signs, fixing any $\underline{j}$ and $\sigma$:
\[ \sum_{\substack{i=1 \\ i \neq j_1,\dots,j_{\ell}}}^k \left(\frac{v-i+\sigma(i)}{k} \right)= 
\sum_{i=1}^k \left(\frac{v-i+\sigma(i)}{k} \right) - \sum_{\alpha=1}^{\ell} \frac{v-j_{\alpha}+\sigma(j_{\alpha})}{k} = v - \frac{\ell v}{k}= \frac{v(k-\ell)}{k}. \]
We deduce then that the inequality in \eqref{eq:sumqke} is satisfied. \par 
If one does the exact same computations, replacing $v$ with $-u$ and the floor signs with ceiling signs, we will obtain the left-most equality in (\eqref{eq:ordg}). This concludes the proof.
\end{proof}
\end{prop}

The result we want to prove regarding the Newton polygon of a generic $f(z,\lambda)$ is a direct consequence of Proposition \ref{prop:newton}.

\begin{corollary}
\label{cor:triangle}
Let $A$ be a generic semi-infinite banded block Toeplitz matrix with blocks of size $k \times k$ and such that $A_{i,j}=0$ for $i-j<-u$ and $i-j>v$ for some $u, v \in \nat$. Then the Newton polygon of its characteristic function $f(z,\lambda)$ is the triangle with vertices $(-u,0)$, $(v,0)$ and $(0,k)$.
\begin{proof}
To a monomial in $f(z,\lambda)$, of the type $ a_{d,e} z^d \lambda^e$, we associate the point $(d,e)$ in $\integ^2$. Then we want to show that the convex hull of the set of non-zero coefficients of $f(z,\lambda)$ is a triangle with vertices $(-u,0)$, $(v,0)$ and $(0,k)$. If we use $(d,e)$ as coordinates on $\integ^2$, this triangle is defined by the inequalities:
\[ \begin{cases}
e \geq 0 \\
k d + v e \leq v k \\
kd - u e \geq -u k
\end{cases}. \]
The first of these inequalities is verified immediately since $f(z,\lambda)=\det(B(z)-\lambda I_k)$. To verify the other two inequalities it is sufficient to notice how Proposition \ref{prop:newton} implies $\frac{-u(k-e)}{k} \leq d \leq \frac{v(k-e)}{k}$. 
\end{proof}
\end{corollary}

\begin{remark}
For a generic symbol $B(z)$ (i.e. if there are no coefficient cancellations which modify the degree and order of $b(z)$), we will have $p=ks$ and $q=kr$. This together with Corollary \ref{cor:triangle} tells us that the hypothesis of Theorem \ref{thm:main} holds generically.
\end{remark}

\medskip

For a fixed $z$ the equation $f(z,\lambda)=0$ has $k$ solutions with respect to $\lambda$. Let us denote them by $\lambda_j(z)$, $j=1,\dots,k$.
In order to prove our main result, we need to study the behaviour of $\lambda_j(z)$ as $z$ approaches $0$ and $\infty$. In order to do this, we give a condition on the coefficients of a two variable polynomials so that its solutions with respect to one variable diverge as the other variable tends to $\infty$. A more general statement appears as Proposition 5.2 in \cite{AlBrSh} without explicit proof. We state a weaker version tailored to our setting, providing a proof for completeness.

\begin{prop}
\label{prop:infty}
Let $h(x,y) \in \complex[x,y]$ be a two-variable polynomial. Denote $m= \deg_y h$, then $h(x,y)$ can be written as
\[ h(x,y) = \sum_{j=0}^m h_j(x) y^j. \]
Assume that $h_0(x) \not\equiv 0$ and 
\begin{align*}
\deg h_0 > \deg h_j && \text{for all } j \in \{1,\dots,m\}.
\end{align*}
Then every function $y(x)$ defined by $h(x,y(x))=0$ satisfies $| y(x) | \rightarrow \infty$ as $ \vert x \vert \rightarrow + \infty$.
\begin{proof}
Let us fix $R>0$ and a function $y(x)$ defined by the equation $h(x,y(x))=0$, and assume $|y(x)| \leq R$. Then, since
\[ h_0(x) = - \sum_{j=0}^m h_j(x) y(x)^j, \]
we have
\[ | h_0(x) | \leq \sum_{j=0}^m |h_j(x)| |y(x)|^j \leq \sum_{j=0}^m |h_j(x)| R^j. \]
Since, by hypothesis, the degree of the polynomial on the left-hand side is greater than the degree of the polynomial on the right-hand side, there exists a $x_R \in \real$ such that, for $x \geq x_R$, 
\[ | h_0(x) | > \sum_{j=0}^m |h_j(x)| R^j. \]
This implies that $|y(x)|>R$ for $x \geq x_R$. Since $R$ was chosen arbitrarily, $|y(x)| \rightarrow \infty$.
\end{proof}
\end{prop}

\begin{lemma}
\label{lmm:infty}
Assume the hypotheses of Theorem \ref{thm:main} are satisfied. For all $j=1,\dots,k$, the function $z \mapsto \lambda_j(z)$ is continuous and maps a neighbourhood of $0$ and a neighbourhood of $\infty$ onto a neighbourhood of $\infty$.
\begin{proof}
The fact that the Newton polygon of $f(z,\lambda)$ is a triangle with vertices $(-q,0)$, $(p,0)$ and $(0,k)$ tells us that, if we write
\[ f(z,\lambda) = \sum_{\ell=0}^k g_{\ell}(z) \lambda^{\ell}, \]
we have that for all $\ell \in \{1,\dots,k\}$
\begin{align*}
\deg g_0 > \deg g_\ell && \text{and} && \operatorname{ord} g_0 < \deg g_\ell.
\end{align*}

So if we apply Proposition \ref{prop:infty} to the bivariate polynomial $z^q f(z,\lambda)$, we obtain that for all $j=1,\dots,k$ $|\lambda_j(z)| \rightarrow \infty$ as $\vert z \vert \rightarrow \infty$. \par 
For $z \rightarrow 0$, we can make the change of variables $w=1/z$ and apply Proposition \ref{prop:infty} to the bivariate polynomial $w^p \widetilde{f}(w, \lambda)$. We will then obtain that $|\lambda_j(z)| \rightarrow \infty$ as $|w| \rightarrow \infty$, or equivalently as $|z| \rightarrow 0$, for all $j=1, \dots,k$.
\end{proof}
\end{lemma}

The last result we need in order to prove our main Theorem has been proven by Delvaux in Corollary 5.2 of \cite{De}. In his result he assumes a restriction on $\lambda$ which is not necessary for us, so we prove an identical statement without the assumption on $\lambda$.

\begin{lemma}
\label{lem:delvaux}
Assume $B(z)$ is a matrix-valued Laurent polynomial, and $z_j(\lambda)$ is defined as in \eqref{eq:zjl}. Then we have that
\[ \begin{cases} 
z_j(\lambda) \rightarrow 0 & j=1,\dots,q \\
z_j(\lambda) \rightarrow \infty & j= q+1, \dots, q+p
\end{cases} \]
for $| \lambda | \rightarrow \infty$.
\begin{proof}
The proof is a straight-forward application of Rouché's theorem together with the structure of $f(z,\lambda)$. We can write
\[ f(z,\lambda) = \sum_{j=-q}^p f_j(\lambda) z^j. \]
We observe that
\[ \deg f_j = \begin{cases}
k & \text{if } j=0 \\
\leq k-1 & \text{if } j \neq 0
\end{cases}. \]
This is a direct consequence of the fact that $f(z,\lambda) = \det(B(z) - \lambda I)$. Now we fix $R>0$ and we consider that the polynomial
\[ z^q f(z,\lambda) = z^q f_0(\lambda) + \sum_{\substack{j=-q \\ j \neq 0}}^p f_j(\lambda) z^{j+q}. \]
Since
\[ \deg f_0(\lambda) > \deg \sum_{\substack{j=-q \\ j \neq 0}}^p f_j(\lambda) R^j, \]
there exists a $l_R>0$ such that, for $|\lambda|> l_R$ and all $z$ with $|z|=R$, we have
\[ |f_0(\lambda)| > \sum_{\substack{j=-q \\ j \neq 0}}^p | f_j(\lambda) | |z|^j. \]
The same inequality holds if we multiply both sides by $z^q$, so by Rouché's theorem $z^q f(z,\lambda)$ has the same number of zeros in the ball $\{ |z| < R\}$ as $z^q f_0(\lambda)$. Of course the latter has exactly $q$ zeros in the ball, all at $z=0$. Since $z^q f(z,\lambda)$ has degree $p+q$ in $z$, there remain $p$ zeros in $\{ |z| > R \}$ for $|\lambda|>l_R$. \par 
To conclude, we apply the reasoning above to two numbers $R=\delta>0$ and $R=M>0$ and take $l=\max(l_\delta,l_M)$, obtaining that $f(z,\lambda)$ has exactly $q$ zeros in $\{|z|<\delta\}$ and $p$ zeros in $\{ |z| > M\}$. It is enough to take $\delta \rightarrow 0$ and $M \rightarrow \infty$ to conclude the proof.
\end{proof}
\end{lemma}

\medskip 

We are now ready to prove our main result.

\begin{proof}[Proof of Theorem \ref{thm:main}]
As we mentioned in the beginning of this section, $\Gamma(B)$ is a real algebraic curve, hence it can not have endpoints in $\complex \setminus \{0\}$. Thus to prove that it contains $k$ Jordan curves with $0$ in their interior it suffices to prove that every path in the Riemann sphere $S^2$ connecting $0$ and $\infty$ intersects $\Gamma(B)$ at least $k$ times. \par 
Define the following sets:
\[ \mathcal{N}_j := \{z \in \complex: \vert z_q(\lambda_j(z)) \vert = \vert z_{q+1}(\lambda_j(z)) \vert \} \]
for $j=1,\dots,k$ (where $z_1(\lambda),\dots,z_{q+p}(\lambda)$ are as in \eqref{eq:zjl}).
Now we want to show that any path connecting $0$ and $\infty$ has a nonempty intersection with each of the $\mathcal{N}_j$'s. After that we will show that each of the $\mathcal{N}_j$'s is contained in $\Gamma(B)$.  \par
Fix $j \in \{1,\dots,k\}$ and let $\gamma:[0,1] \rightarrow S^2$ denote a path in $S^2$ connecting $0$ and $\infty$. By construction, for any $t \in [0,1]$ the value $\vert \gamma(t) \vert$ appears at least once among $\vert z_1(\lambda_j(\gamma(t)) \vert, \dots, \vert z_{p+q}(\lambda_j(\gamma(t)) \vert$. This is because $z=\gamma(t)$ is a root of $z^q f(z,\lambda_j(z))$. Since $\gamma(0)=0$, Lemma \ref{lmm:infty} together with Lemma \ref{lem:delvaux} tell us that for $t$ in a right-neighbourhood of $0$ the value $\vert \gamma(t) \vert$ appears among the first $q$ values $\vert z_1(\lambda_j(\gamma(t)) \vert, \dots, \vert z_q(\lambda_j(\gamma(t)) \vert$. The same results tell us that for $t$ in a neighbourhood of $\infty$ the value $|\gamma(t)|$ will appear among the remaining $p$ values. We claim that there must exist a $t_0 \in (0,1)$ such that

\begin{equation}
\label{eq:3eq}
\vert \gamma(t_0) \vert = \vert z_q(\lambda_j(\gamma(t_0)) \vert = \vert z_{q+1}(\lambda_j(\gamma(t_0))) \vert.
\end{equation}

The reasoning behind this claim is as follows: we know that near 0
\[ |\gamma(t)| \leq |z_q(\lambda_j(\gamma(t))| \]
and near $\infty$
\[ \gamma(t) \geq |z_{q+1}(\lambda_j(\gamma(t))| \geq |z_q(\lambda_j(\gamma(t))|. \]
Let us denote 
\[t_0 = \sup \{t: |\gamma(t)| \leq |z_q(\lambda_j(\gamma(t))| \}. \]
Then for $t>t_0$, we have $|\gamma(t)| > |z_q(\lambda_j(\gamma(t)|$, but since $\gamma(t)$ has to be among the solutions, we have that $|\gamma(t)| \geq |z_{q+1}(\lambda_j(\gamma(t)|$. By continuity of $\gamma(t)$, we have
\begin{align*}
|\gamma(t_0)| \leq |z_q(\lambda_j(\gamma(t_0))| && \text{ and } && |\gamma(t_0)| \geq |z_{q+1}(\lambda_j(\gamma(t_0))| \geq |z_q(\lambda_j(\gamma(t_0))|
\end{align*}

so \eqref{eq:3eq} follows. This implies that $\mathcal{N}_j \cap \gamma([0,1]) \neq \emptyset$.
\par 
We have showed that each one of the sets $\mathcal{N}_j$ intersects any path connecting $0$ and $\infty$. Once again, fix $j\in\{1,\dots,k\}$. By definition of $\mathcal{N}_j$ and $\Lambda_0(B)$, we have that $\lambda_j(\mathcal{N}_j) \subset \Lambda_0(B) \subset \real$. Let $z_0$ be such that $\lambda_j(z_0) \in \real$, then we have $f(z_0,\lambda_j(z_0))=0$ with $\lambda_j(z_0) \in \real$; this implies that $z_0 \in \Gamma(B)$. Thus we have that each one of the $\mathcal{N}_j$'s is contained in $\Gamma(B)$. This proves that $\Gamma(B)$ intersects any path connecting $0$ and $\infty$ at least $k$ times, as we wanted.
\end{proof}

We don't have a proof for the other implication of Conjecture \ref{conj:main}, not even for $k=1$ (i.e. the Toeplitz case), but there are various numerical experiments which corroborate the conjecture.

\section{Numerical Experiments}

In this section we present some numerical experiments which illustrate the conclusion of Theorem \ref{thm:main} and show why we think that Conjecture \ref{conj:main} might be true. For very interesting experiments regarding the Toeplitz case (i.e. $k=1$) we refer to \cite{ShSt}. All of our examples regard the case $k>1$. \par 
For the computations and figures of these examples we used the software Wolfram Mathematica.

\begin{figure}[t]
\centering
\begin{minipage}{0.45\textwidth}
\centering
\includegraphics[width=0.9\textwidth]{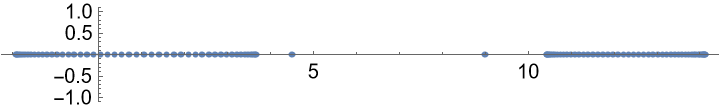}
\caption{Eigenvalues of matrix $T_n(B_1)$ with size $n=100$}
\label{fgr:eigen1}
\end{minipage}
\begin{minipage}{0.5\textwidth}
\centering
\includegraphics[width=0.9\textwidth]{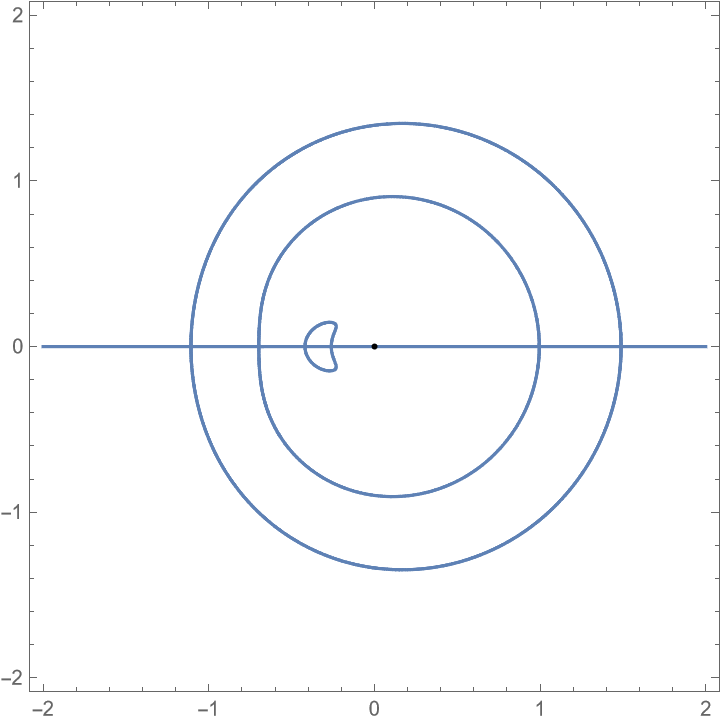}
\caption{Plot of $\Gamma$ for the matrix $A_1$}
\label{fgr:gamma1}
\end{minipage}
\end{figure}

\begin{figure}[b]
\centering
\begin{minipage}{0.45\textwidth}
\centering
\includegraphics[width=0.9\textwidth]{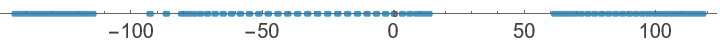}
\caption{Eigenvalues of matrix $T_n(B_2)$ with size $n=100$}
\label{fgr:eigen2}
\end{minipage}
\begin{minipage}{0.5\textwidth}
\centering
\includegraphics[width=0.9\textwidth]{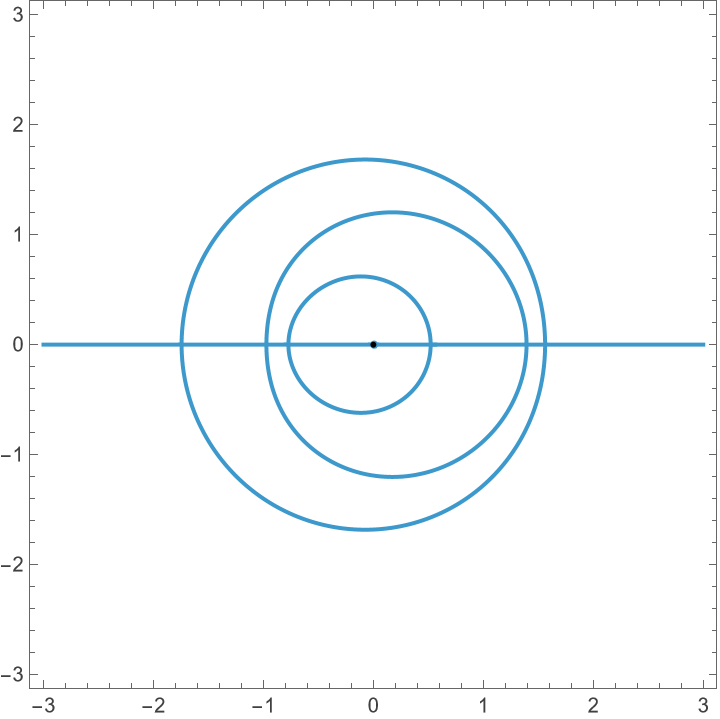}
\caption{Plot of $\Gamma$ for the matrix $A_2$}
\label{fgr:gamma2}
\end{minipage}
\end{figure}

\begin{example}
The first example we give is one of a non-Hermitian matrix with all real eigenvalues. \par 
We consider the following 4-diagonal block Toeplitz matrix with blocks of size 2:
\[ A_1= \begin{pmatrix}
8 & -5 & -2 &  & & \\
-2 & 5 & -4 & 1 & & \\
 & -6 & 8 & -5 & -2 & \\
 &  & -2 & 5 & -4 & \ddots \\
 & & & \ddots & \ddots & \ddots
\end{pmatrix}. \]
In this case $p=1$, $q=2$ and $k=2$. The symbol of $A_1$ is
\[ B_1(z) = \begin{pmatrix}
-2 z^{-1}+8 & -5-6z \\
-4z^{-1}-2 & z^{-1}+5
\end{pmatrix}. \]
In Figure \ref{fgr:eigen1} we plot the eigenvalues of $T_{100}(B_1)$ with a precision of 50 decimals; one can see that they are all real. To plot the set
\[ \Gamma(B_1) = \{z \in \complex: \exists \lambda \in \real \text{ s.t. } f(z,\lambda)=0 \} \]
we compute 
\[ f(z,\lambda) = 6-\frac{2}{z^2}-\frac{22}{z}-12z-13 \lambda + \frac{\lambda}{z}+\lambda^2. \]
By computing also the real and imaginary part and eliminating one of the variables, we obtain $\Gamma(B)=\{z=x+iy : g(x,y)=0\}$ with
\begin{multline*}
g(x,y) = 72 x^8 y + 288 x^6 y^3 + 432 x^4 y^5 + 288 x^2 y^7 + 72 y^9 - 
 186 x^6 y - 558 x^4 y^3 - 558 x^2 y^5 - \\
 186 y^7 - 60 x^5 y - 
 120 x^3 y^3 - 60 x y^5 + 102 x^4 y + 204 x^2 y^3 + 102 y^5 + 
 62 x^3 y + 62 x y^3 + 9 x^2 y + y^3.
\end{multline*}
Its plot is shown in Figure \ref{fgr:gamma1}. One can observe that three Jordan curves are present, two circling the origin and one not. This is consistent with the result in Theorem \ref{thm:main}. \par 
In our second example we consider $k=3$ and the 4-diagonal block Toeplitz matrix $A_2=T(B_2)$ with
\[ B_2(z) = \begin{pmatrix}
-31 & 49 & 1-45z \\
98-31 z^{-1} & 20 & 58 \\
-94 z^{-1} & 51-13 z^{-1} & -62
\end{pmatrix}. \]
As we did before, we plot $\sp(T_{100}(B_2))$ and $\Gamma(B_2)$ in Figure \ref{fgr:eigen2} and \ref{fgr:gamma2}, respectively. One can notice once again how the figures are consistent with our conjecture.
\end{example}

\begin{figure}[b]
\centering
\begin{minipage}{0.45\textwidth}
\centering
\includegraphics[width=0.9\textwidth]{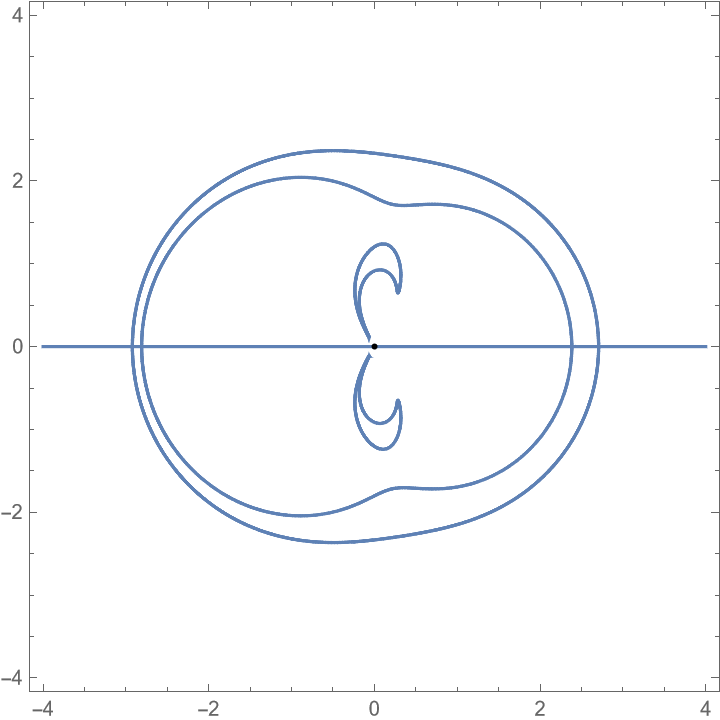}
\caption{$\zeta=70$}
\end{minipage}
\begin{minipage}{0.45\textwidth}
\centering
\includegraphics[width=0.9\textwidth]{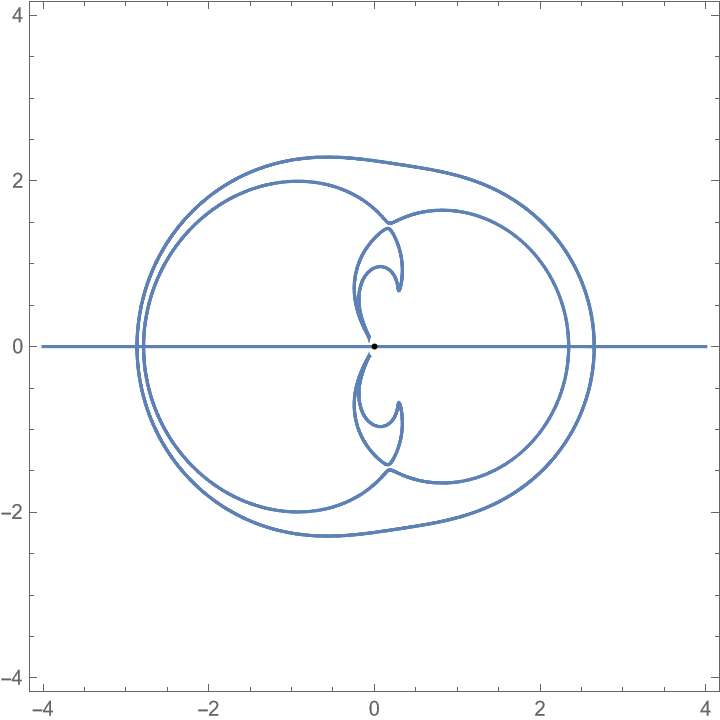}
\caption{$\zeta=47$}
\end{minipage}
\begin{minipage}{0.45\textwidth}
\centering
\includegraphics[width=0.9\textwidth]{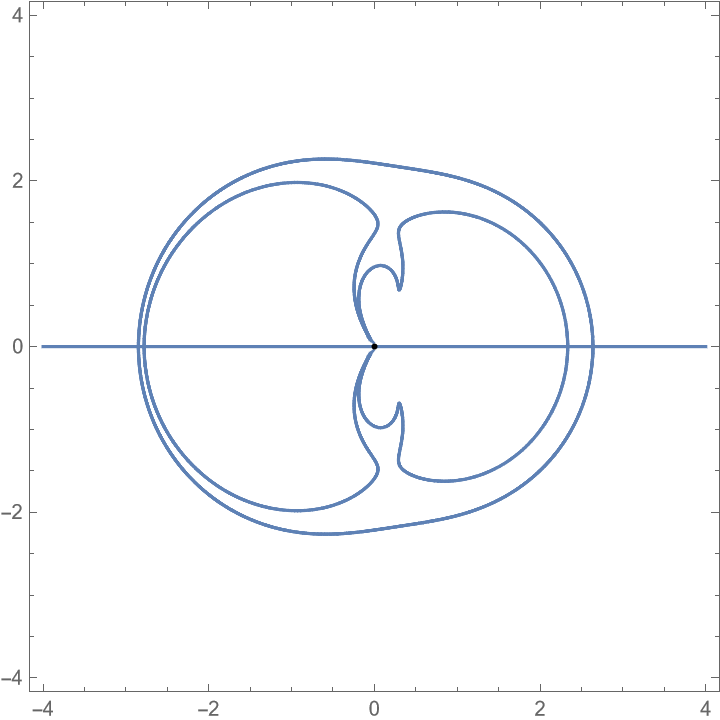}
\caption{$\zeta=40$}
\end{minipage}
\begin{minipage}{0.45\textwidth}
\centering
\includegraphics[width=0.9\textwidth]{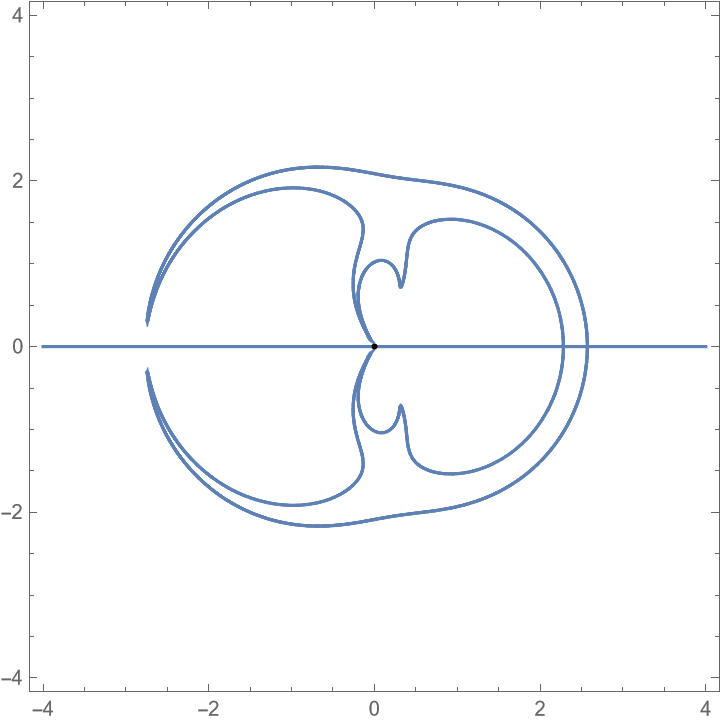}
\caption{$\zeta=10$}
\end{minipage}
\end{figure}

\begin{example}
In this second example we show how if the asymptotic spectrum of $A$ is not real, the condition on $\Gamma$ containing $k$ Jordan curves with the origin in their interior no longer holds. To do this, we consider the 5-diagonal Hessenberg block Toeplitz matrix $A_3$ with symbol
\[ B_3(z) = \begin{pmatrix}
19 z^{-1}+ \zeta & -57 z^{-1}+61-86z \\
-91z^{-2}-65z^{-1}-13 & -9z^{-1}-40
\end{pmatrix}. \]
In the Figures we can see the plot of $\Gamma$ for various values of $\zeta$. For $\zeta=70$ and $\zeta=47$ the asymptotic spectrum of $A$ is included in the real line, while for $\zeta=40$ and $\zeta=10$ it is not. One can see how the plot of $\Gamma$ gets modified and loses the property mentioned in Conjecture \ref{conj:main}.
\end{example}

\begin{acknow}
I want to thank my supervisor, Professor B. Shapiro for introducing me to this subject and for his endless patience. There is no doubt that his insight and knowledge have been essential for this work to come to life.
\end{acknow}


\begin{thebibliography}{99}

\bibitem{AlBrSh} P.~ Alexandersson, P.~Brändén, and B.~Shapiro. \emph{An inverse problem in Pólya-Schur theory. I. Non-degenerate and degenerate operators}, Revista matemática iberoamericana, Vol. 41, no 5, p. 1863-1896 (2025).

\bibitem{ABL} H.~Ammari, S.~Barandun, and  P.~Liu, \emph{Perturbed block Toeplitz matrices and the non-Hermitian skin effect in dimer systems of subwavelength resonators}. Journal de Mathématiques Pures et Appliquées, 195, p.103658 (2025). 


\bibitem{ABdBLT} H.~Ammari, S.~Barandun, Y.~De Bruijn, P.~Liu, and C.~Thalhammer, \emph{Spectra and pseudo-spectra of tridiagonal k-Toeplitz matrices and the topological origin of the non-Hermitian skin effect}. Journal of Physics A: Mathematical and Theoretical, 58(20), p.205201 (2025).

\bibitem{CZ} Haoyan Chen,  and Yi Zhang, \emph{Real spectra in one-dimensional single-band non-Hermitian Hamiltonians}. Physical Review B, 108(11), p.115114 (2023).

\bibitem{BPP} N.~Bebiano, J.~Da Provid\^encia and J.~P.~Da Provid\^encia, \emph{Classes of non-hermitian operators with real eigenvalues}, Electronic Journal of Linear Algebra,  Volume 21, pp. 98-109 (2010). 

\bibitem{BoeGr} A.~Böttcher and S. M. Grudsky, \emph{Spectral properties of banded Toeplitz matrices}, Society for Industrial and Applied Mathematics (2005).

\bibitem{De} S.~Delvaux,  \emph{Equilibrium problem for the eigenvalues of banded block Toeplitz matrices}, Math. Nach. Volume 285, Issue 16, Pages 1935-1962 (2012). 

\bibitem{KS} K.~Kawabata, and M.~Sato, \emph{Real spectra in non-Hermitian topological insulators}, Phys. Rev. Res. 2, 033391 (2020). 

\bibitem{ShSt} B.~Shapiro, F.~\v Stampach,  \emph{Non-selfajoint Toeplitz matrices whose principal submatrices have real spectrum},  Constr. Approx. 49 , no. 2, 191--226, (2019)

\bibitem{ShSt2} B.~Shapiro, F.~\v Stampach,  \emph{Errata to "Non-self-adjoint Toeplitz matrices whose principal submatrices have real spectrum"},  Constructive Approximation (2023).

\bibitem{SnSp} P.~Schmidt, F.~Spitzer, \emph{The Toeplitz matrices of an arbitrary Laurent polynomial}, Mathematica Scandinavica 8 , 15–38 (1960). 


\bibitem{Wi1} H.~Widom, \emph{Asymptotic behavior of block Toeplitz matrices and determinants}, Adv. Math., 13, 284-322 (1974).

\bibitem{Wi2} H.~Widom,  \emph{Asymptotic behavior of block Toeplitz matrices and determinants. II}, Adv. Math., 21, l-29 (1976).


\bibitem{YLFWZ} G.~Yang, Y.-K.~Li, Y.~Fu , Zh.~Wang, and Y.~Zhang, \emph{Complex semiclassical theory for non-Hermitian quantum system}, Phys.  Rev. 109, 045110 (2024). 

\end{thebibliography}
\end{document}